\documentclass[11pt]{amsart}
\usepackage{amsmath,amsfonts}
\usepackage{amssymb}
\usepackage{amscd}
\usepackage{amsthm}
\usepackage{yhmath}
\usepackage{subfigure}

\usepackage[all]{xy}
\usepackage{color}

\numberwithin{equation}{section}

\theoremstyle{plain}
\newtheorem{theorem}{Theorem}[section]
\newtheorem*{theorem*}{Theorem}
\newtheorem{lemma}[theorem]{Lemma}
\newtheorem{corollary}[theorem]{Corollary}
\newtheorem{proposition}[theorem]{Proposition}
\theoremstyle{remark}
\newtheorem{remark}[theorem]{Remark}
\numberwithin{equation}{section}
\theoremstyle{definition}

\numberwithin{equation}{section}

\newcommand{\NN}{\Bbb{N}}

\newcommand{\RR}{\Bbb{R}}
\newcommand{\CC}{\Bbb{C}}

\newcommand{\MH}{Mihlin--H\"ormander }

\newcommand{\vg}{\mathfrak{v}}

\newcommand{\zg}{\mathfrak{z}}

\newcommand{\di}{\,{\rm{d}}}

\newcommand{\nep}{{\rm{e}}} 

\begin{document}

\title[Spectral multipliers for Laplacians]
{Spectral multipliers for Laplacians\\ with drift on Damek--Ricci spaces}


\author{Alessandro Ottazzi}
\address{
CIRM Fondazione Bruno Kessler, Via Sommarive 14, 38123 Trento, Italy}
\email{ottazzi@fbk.eu}

\author{Maria Vallarino}
\address{Dipartimento di Scienze Matematiche Giuseppe Luigi Lagrange, Politecnico di Torino, Corso Duca degli Abruzzi 24, 10129 Torino, Italy}
\email{maria.vallarino@polito.it}


\thanks{Work partially supported by Progetto GNAMPA 2012 ``Analisi armonica su variet\`a, spazi di Wiener e gruppi di Lie'' and by Progetto GNAMPA 2013 ``Analisi armonica e geometrica su gruppi di Lie e spazi
subRiemanniani''}


\maketitle 

\begin{abstract}
We prove a multiplier theorem for certain Laplacians with drift on Damek--Ricci spaces, which are a class of Lie groups of exponential growth. Our theorem  
generalizes previous results obtained by W. Hebisch, G. Mauceri and S. Meda on Lie groups of polynomial growth.  

\end{abstract}

\section{Introduction}\label{intro}

W. Hebisch, G. Mauceri and S. Meda \cite{HMM} studied spectral multipliers of right invariant sub-Laplacians with drift on a noncompact connected Lie group $G$. The operators they consider are self--adjoint with respect to a positive measure, whose density with respect to the left Haar measure is a nontrivial positive character of $G$. If $G$ is amenable, they showed that every $L^p$ spectral multiplier of such sub-Laplacians with drift extends to a bounded holomorphic function on a parabolic region in the complex plane. When $G$ is of polynomial growth they proved that this necessary condition is nearly sufficient, by proving that bounded holomorphic functions on a suitable parabolic region which satisfy certain regularity conditions are spectral multipliers of such operators. However, if $G$ has exponential growth, the question of finding a sufficient condition for a function to be a $L^p$ multiplier of self-adjoint sub-Laplacians with drift remains open. In this note we contribute to this problem by considering certain Laplacians with drift on a class of Lie groups of exponential growth, namely the harmonic extensions of $H$-type groups.
 
Before going into more detail, we should mention that sub-Laplacians with drift were studied by other authors in the literature. Heat kernel estimates for sub-Laplacians with drift were studied on various Lie groups in \cite{Al2, D, VSC}. N. Lohouh\'e and S. Mustapha \cite{LM} studied the $L^p$ boundedness of the Riesz transforms of any order associated with sub-Laplacians with drift on every amenable Lie group. In particular, they analysed in \cite[Section IV]{LM}  the case of harmonic extensions of $H$-type groups. 
 
\smallskip

Let now $(\mathcal X,\mu)$ be a measure space and $T$ a linear nonnegative self-adjoint operator on $L^2(\mu)$. Let $\{E(\lambda)\}$ denote the spectral resolution of the identity 
for which $T = \int_0^{\infty}\lambda\di E(\lambda)$. 
By the spectral theorem, if $M$ is a bounded Borel measurable function on 
$[0,\infty)$, then the operator $M(T)$ defined by
$$
M(T)=\int_0^{\infty}M(\lambda)\di E(\lambda)   
$$
is bounded on $L^2(\mu)$. We call $M$ a $L^p$ {\emph{spectral multiplier}} for $T$, $p\in [1,\infty)\setminus\{2\}$, if $M(T)$ extends to bounded operator on $L^p(\mu)$. The spectral multiplier problem for $T$ consists in finding conditions, necessary or sufficient, for a function $M$ to be a $L^p$ multiplier for $T$. 

We say that an operator $T$ admits a $L^p$ \emph{holomorphic functional calculus}
if every $L^p$ multiplier of $T$ extends to an holomorphic function
on some neighbourhood of its $L^2$ spectrum. By contrast, an operator $T$ is said to admit a $L^p$ \emph{differentiable functional calculus} 
if a function which satisfies suitable differentiable conditions on $[0,\infty)$ is 
a $L^p$ multiplier for $T$.

\smallskip

The functional calculus for Laplacians on Lie groups has been intensively studied. Let $G$ be a connected noncompact Lie group of topological dimension $n$. Let $X_0,...,X_{n-1}$ be a basis of left invariant vector fields on $G$ and let
$$
L_0=-\sum_{i=0}^{n-1}X_i^2
$$ 
be the corresponding {\emph{Laplacian}}, which is nonnegative and essentially self-adjoint on $L^2(\rho)$, where $\rho$ is the right Haar measure of $G$. The functional calculus for $L_0$ has been the object of investigation in different classes of Lie groups.  If the group $G$ has polynomial growth, then $L_0$ has differentiable functional calculus \cite{Al1}.
 When $G$ is of exponential growth the situation changes. 
There are classes of Lie groups of exponential growth  and Laplacians $L_0$ which admit $L^p$ differentiable functional calculus \cite{CGHM, H, HS}, and others which  admit a $L^p$ holomorphic functional calculus \cite{CMu, HLM, LuMu, LuMuS}. We refer the reader to \cite{CGHM, HMM} and the references therein for a more detailed discussion. 

\smallskip

Operators of the form 
$$
L_X=L_0-X\,,
$$
with $X$ a left invariant vector field, are called {\emph{Laplacians with drift}} and they also have been studied by various authors.  It turns out that $L_X$ is symmetric with respect to some measure if and only if there exists a nontrivial positive character $\chi$ such that $X=\sum_{i=0}^{n-1}X_i(\chi)(e)X_i$ \cite{HMM, LM}, where $e$ denotes the identity of $G$. In this case $L_X$ is essentially self-adjoint on $L^2(\chi \di\rho)$ and one can study its functional calculus. If $G$ is amenable, then the results of Hebisch, Mauceri and Meda can be reformulated in terms of left invariant vector fields and right measure and imply that $L_X$ has $L^p$ holomorphic functional calculus. When $G$ has polynomial growth, they also established a sufficient condition for $L^p$ multipliers of $L_X$. To the best of our knowledge, it remains unknown if a similar sufficient condition holds for exponential growth Lie groups, where the technics in \cite{HMM} do not seem to apply. As we mentioned at the beginning of this section, we consider here the case of harmonic extensions of H-type groups, also called 
{\it Damek--Ricci spaces}. Such groups were 
introduced by E.~Damek and F.~Ricci \cite{D1}, \cite{D2}, \cite{DR1}, \cite{DR2},
and include all rank one symmetric spaces of the noncompact type.
Most of them are nonsymmetric harmonic manifolds,  
and provide counterexamples to the Lichnerowicz conjecture. The geometry of these spaces was studied by M.~Cowling, 
A.~H.~Dooley, A.~Kor\'anyi and Ricci in \cite{CDKR1}, \cite{CDKR2}.
Given an $H$-type group $N$, let $S=NA$ be the one-dimensional extension of $N$ obtained by letting $A=\Bbb{R}^+$ act on $N$ by homogeneous dilations. The group $S$ is solvable, hence amenable, nonunimodular and it is a Lie group of exponential growth. Let $L_0$ be a distinguished left-invariant Laplacian on $S$ (see Section 3 for its precise definition). The operator $L_0$ has a $L^p$ differentiable functional calculus \cite{A, HS, M, V}.  Our main result, Theorem \ref{main}, concerns spectral multipliers of $L_X$, where $X$ is a drift such that $L_X$ is symmetric. We prove that for every $p$ in $(1,\infty)\setminus \{2\}$, a  function $M$ which is holomorphic in a parabolic region depending on the drift and on $p$ and which satisfies suitable regularity conditions on its boundary is a $L^p$ spectral multiplier of $L_X$. Our result generalizes the one proved in \cite{HMM} for polynomial growth Lie groups to Damek--Ricci spaces. We prove it by splitting the kernel of the multiplier operator into a local and a global part: the analysis of the local part follows the methods in \cite{HMM}, while the analysis of the global part requires a different proof and is based on spherical analysis. 

\smallskip
We conclude the introduction indicating some further questions that may be addressed. On the one hand, it would be interesting to extend the results obtained in this paper to subLaplacians with drift on Damek--Ricci spaces. On the other hand, one could also consider $NA$ groups 
coming from the Iwasawa decomposition of a noncompact semisimple Lie group of arbitrary rank and study $L^p$ multipliers of the 
Laplacians with drift $L_0-X$, where $L_0$ is the Laplacian studied in \cite{CGHM} and $X$ is a left invariant vector field such that $L_X$ is symmetric. This will be the object of further investigation.

\smallskip
 
The paper is organized as follows: in Section~2, we recall the definition of $H$-type group $N$ and its Damek--Ricci extension~$S$; then we summarize some results of spherical analysis on~$S$. In Section~3 we introduce the Laplacians with drift on Damek--Ricci spaces and prove our multiplier theorem.

\smallskip
 
Throughout the article the expression $A\lesssim B$ means that there exists a positive constant $C$ such that $A\leq C\,B\,.$

\section{Preliminaries}
We recall the definition of $H$-type groups, describe their Damek--Ricci extensions, and collect some results on spherical analysis on these spaces that we shall use. For more details the reader is referred to  \cite{A1, A, ADY, CDKR1, CDKR2, D1, D2, DR1, DR2, F, R}.
\smallskip

Let $\frak{n}$ be a Lie algebra equipped  with an inner product $\langle\cdot,\cdot\rangle$. Let $\frak{v}$ and $\frak{z}$ be complementary orthogonal subspaces of $\frak{n}$ such that $[\frak{n},\frak{z} ]=\{0\}$ and $[\frak{n},\frak{n}]\subseteq \frak{z}$.
According to Kaplan \cite{K}, the algebra $\frak{n}$ is of $H$-type if for every $Z$ in $\frak{z}$ of unit length the map $J_Z:\frak{v}\to \frak{v}$, defined by
$$\langle J_ZX,Y\rangle\,=\,\langle Z,[X,Y]\rangle\qquad\forall X, Y\in \frak{v}\,,$$
is orthogonal. The connected and simply connected Lie group $N$ associated to $\frak{n}$ is called  $H$-type group. We identify $N$ with its Lie algebra $\frak{n}$ via the exponential map
\begin{eqnarray*}
\frak{v}\times\frak{z} &\to& N\\
(X,Z)&\mapsto& \exp(X+Z)\,.
\end{eqnarray*}
Set $Q={(m_{\frak{v}}+2m_{\frak{z}})}/{2}$, where $m_{\frak{v}}$ and $m_{\frak{z}}$ are the dimensions of $\frak{v}$ and $\frak{z}$, respectively.

Let $S$ be the one-dimensional extension of $N$ obtained as semidirect product with $A=\Bbb{R}^+$, which acts  
on $N$ by homogeneous dilations. Let $H$ denote a vector in $\frak{a}$ acting on
$\frak{n}$ with eigenvalues $1/2$ and (possibly) $1$; we extend the inner product 
on $\frak{n}$ to the algebra $\frak{s}=\frak{n}\oplus\frak{a}$, by requiring $\frak{n}$
and $\frak{a}$ to be orthogonal and $H$ to be a unit vector.  The map
\begin{eqnarray*}
\frak{v}\times\frak{z}\times\mathbb{R}^+ &\to& S\\
(X,Z,a)&\mapsto& \exp(X+Z)\exp(\log a \,H)
\end{eqnarray*}
gives global coordinates on $S$. 
The
group $S$ is nonunimodular: the right and left Haar measures on
$S$ are  given by 
$${\mathrm{d}}\rho(X,Z,a)=a^{-1}\,{\mathrm{d}} X\,{\mathrm{d}} Z\,{\mathrm{d}} a
\qquad{\textrm{and}}\qquad
{\mathrm{d}}\mu(X,Z,a)=a^{-(Q+1)}\,{\mathrm{d}} X\,{\mathrm{d}} Z\,{\mathrm{d}} a \,.$$  
Therefore the modular function is $\delta(X,Z,a)=a^{-Q}$. 
We equip $S$ with the left invariant Riemannian metric which agrees with the inner product on $\frak{s}$ at the identity $e$. We denote by $d$ the distance induced by this Riemannian structure. For every point $x\in S$ we write $r(x)=d(x,e)$ and we denote by $a(x)$ the component along $\RR^+$ of $x$.

We may identify $S$ with the open unit ball $\mathcal B$ in $\mathfrak s$
$$\mathcal B=\{(X,Z,t)\in\vg \times \zg \times \RR: ~\|(X,Z,t)\|=|X|^2+|Z|^2+t^2<1\}\,,$$
via the bijection (see \cite{CDKR2}) $F:S\to \mathcal B$ defined by 
\begin{align*}
F(X,Z,a)=\frac{1}{\Big(1+a+\frac{1}{4}|X|^2\Big)^2+|Z|^2}\Big(&\Big(1+a+\frac{1}{4}|X|^2-J_Z\Big)X,2Z, \\
&-1+\Big(a+\frac{1}{4}|X|^2\Big)^2+|Z|^2\Big)\,.
\end{align*}
The left Haar measure on $S$ may  be normalized in such a way that for all functions $f$ in $C_c^{\infty}(S)$ 
$$\int _S f\di \mu=\int_0^{\infty}\int_{\partial \mathcal B}f\big(F^{-1}(r\omega)\big)A(r)\di r\di \sigma(\omega)\,,$$
where $\di \sigma$ is the surface measure on $\partial \mathcal B$ and
\begin{equation}\label{misura}
A(r)=2^{m_{\vg}+2m_{\zg}}\sinh ^{m_{\vg}+m_{\zg}}\left(\frac{r}{2}\right)\cosh ^{m_{\zg}}\left(\frac{r}{2}\right)
\qquad \forall r\in \mathbb R^+\,.
\end{equation}
It is easy to check that
\begin{equation}\label{pesoA}
A(r)\lesssim \left(\frac{r}{1+r}\right)^{n-1}\nep^{Qr}\qquad\forall r\in\RR^+\,.
\end{equation}
We say that a function $f$ on the group $S$ is radial if it depends only on the distance from the identity, i.e., if there exists a function $f_0$ defined on $[0,+\infty)$ such that $f(X,Z,a)=f_0(r)$, where $r=d\big((X,Z,a),e\big)$. We abuse the notation and write $f(r)$ instead of $f_0(r)$.  

Damek and Ricci \cite{DR1} defined the radialisation operator
$$\mathcal R: ~C^{\infty}_c(S)\to C^{\infty}_c(S)$$
in the following way:
$$\mathcal R f(x)=\left({\mathcal{\widetilde{R}}}(f\circ F^{-1})\right)\big(F(x)\big)\qquad\forall x\in S,$$
where ${\mathcal{\widetilde{R}}}$ is the radialisation operator on the ball $\mathcal B\,$ defined by
$$({\mathcal{\widetilde{R}}}\phi)(\omega)=\frac{1}{|\partial \mathcal B|}\int_{\partial \mathcal B}\phi(\|\omega\|\omega)\di \sigma(\omega)\,.$$
A function $f$ is radial if and only if $\mathcal R(f)=f$.   

On the Riemannian manifold $S$ we may consider the (positive definite) Laplace--Beltrami operator $\mathcal L$. A radial function $\phi$ on the group $S$ is called {\emph{spherical}} if it is an eigenfunction of $\mathcal L$ and $\phi(e)=1$. One can prove that all spherical functions are given by $\phi_{\lambda}=\mathcal R (a(\cdot)^{-i\lambda+Q/2})$, for $\lambda\in\CC$, and that the eigenvalue corresponding to $\phi_{\lambda}$ is $\lambda^2+Q^2/4$. In \cite[Lemma 1]{A}, it is shown that
\begin{equation}\label{stimafi0}
\phi_0(r)\lesssim \,(1+r)\,{\mathrm{e}}^{-\frac{Qr}{2}}\qquad\forall r\in\mathbb{R}^+\,.
\end{equation}
In the following lemma we shall estimate the modulus of $ \phi_\lambda$.  
\begin{lemma}\label{funzionisferiche}
For every $r\in\RR^+$ and $\lambda\in\CC$
$$
| \phi_\lambda(r) | \lesssim e^{|{\rm{Im}} \lambda | r}(1+r)e^{-\frac{Q}{2}r}\,.
$$
\end{lemma}
\begin{proof}
For all $r$ in $\mathbb R^+$
\begin{align}
\phi_\lambda(r)&={\mathcal R}\big(a(\cdot)^{-i\lambda+\frac{Q}{2}}\big)(x)\nonumber \\
&= \frac{1}{|\partial B|}\int_{\partial B}(a\circ F^{-1})^{-i\lambda}(\|F(x)\| \omega)(a\circ F^{-1})^{\frac{Q}{2}}(\|F(x)\| \omega)d\sigma(\omega)\label{formulaphilambda},
\end{align}
where $x$ is a point in $S$ such that $r(x)=r$.  By \cite[Formula (1.20)]{ADY},  for every $z\in S$ we have that
$$
e^{-r(z)}\leq a(z)\leq e^{r(z)},
$$
which implies 
\begin{equation}\label{adiz}
a(z)^{-i\lambda}\leq e^{|{\rm{Im}} \lambda|r(z)}.
\end{equation}
If $\omega\in \partial B$ and $z=F^{-1}(\|F(x)\|\omega)$, then by \cite[Theorem 1.1]{R}  
$$
r(z)=\frac{\log (1+\|F(x)\|)}{1-\|F(x)\|}=r(x).
$$
Therefore \eqref{adiz} becomes $a(z)^{-i\lambda} \leq e^{|{\rm{Im}} \lambda|r(x)}$.  Substituting in 
\eqref{formulaphilambda}, we obtain
\begin{align*}
|\phi_\lambda(r)|&\leq   \frac{1}{|\partial B|}\int_{\partial B} e^{|{\rm{Im}} \lambda|r(x)} (a\circ F^{-1})^{\frac{Q}{2}}(\|F(x)\| \omega)d\sigma(\omega)\\
&= e^{|{\rm{Im}} \lambda|r(x)}\phi_0(r)\\
&\lesssim e^{|{\rm{Im}} \lambda|r}(1+r)e^{-\frac{Q}{2}r},
\end{align*}
where in the last inequality we used \eqref{stimafi0}. This concludes the proof.

\end{proof}

The {\it spherical Fourier transform} ${\mathcal H} f$ of an integrable radial function $f$ on $S$ is defined by the formula
$${\mathcal H} f(\lambda)=\int_S \phi _{\lambda}\,f\,{\mathrm{d}}\mu  \,.$$
For ``nice'' radial functions $f$ on $S$, we have the following inversion formula 
$$f(x)=c_S\int_{0}^{\infty}{\mathcal H} f (\lambda)\,\phi _{\lambda}(x)\,| {\mathbf{c}} (\lambda) |^{-2} \,{\mathrm{d}} \lambda\qquad \forall x\in S\,,$$
and  the Plancherel formula: 
$$\int_S|f|^2\,{\mathrm{d}}\mu  =c_S\int_0^{\infty}|{\mathcal H} f (\lambda)|^2\,|{\mathbf{c}}(\lambda)|^{-2}\,{\mathrm{d}} \lambda\,,$$
where the constant $c_S$ depends only on $m_{\frak{v}}$ and $m_{\frak{z}}$, and ${\mathbf{c}}$ denotes the Harish-Chandra fun\-ction (see \cite[Section 2]{ADY}). Let ${\mathcal A}$ denote the Abel transform defined for any radial function $f$ on $S$ by
$$\mathcal Af(t)= \int_N f(X,Z,\nep^t)\,\nep^{-Qt/2}\di X\di Z\qquad \forall t\in\RR\,, $$
and let ${\mathcal F}$ denote the Fourier transform on the real line, defined by 
$${\mathcal F} g(s)=\int_{-\infty}^{+\infty}g(r)\,{\mathrm{e}}^{-isr}\,{\mathrm{d}} r\,,$$ 
for each integrable function $g$ on~$\mathbb{R}$. It is well known that ${\mathcal H}={\mathcal F}\circ {\mathcal A}$, hence ${\mathcal H}^{-1}={\mathcal A}^{-1}\circ {\mathcal F}^{-1}$. For later use, we recall the inversion formula for the Abel transform \cite[Formula (2.24)]{ADY}.
If $m_{\frak{z}}$ is even, then
\begin{equation}\label{inv1}
{\mathcal A}^{-1}f(r)=a_S^e\, \Big(-\frac{1}{\sinh r}\,\frac{\partial}{\partial r}\Big)^{m_{\frak{z}}/2}\Big(-\frac{1}{\sinh(r/2)}\,\frac{\partial}{\partial r}\Big)^{m_{\frak{v}}/2}f (r)\,,
\end{equation}
where $a_S^e=2^{-(3m_{\frak{v}}+m_{\frak{z}})/2}\pi^{-(m_{\frak{v}}+m_{\frak{z}})/2}$. On the other hand, if $m_{\frak{z}}$ is odd, then
\begin{equation}\label{inv2}
{\mathcal A}^{-1}f(r)=a_S^o\int_r^{\infty} \Big(-\frac{1}{\sinh s}\,\frac{\partial}{\partial s}\Big)^{(m_{\frak{z}}+1)/2}\Big(-\frac{1}{\sinh(s/2)}\,\frac{\partial}{\partial s}\Big)^{m_{\frak{v}}/2}f (s) \,{\mathrm{d}}\nu(s)\,,
\end{equation}
where $a_S^o= 2^{-(3m_{\frak{v}}+m_{\frak{z}})/2}\pi^{-n/2}$ and ${\mathrm{d}}\nu(s)=(\cosh s-\cosh r)^{-1/2}\sinh s \,{\mathrm{d}} s$.

In the sequel for every $p$ in $(1,\infty)$ we shall denote by $Cv_p(\rho)$ the space of right $L^p(\rho)$ convolutors on $S$, i.e. the space of distributions $k$ on $S$ such that the operator $f\mapsto f\ast k$ is bounded on $L^p(\rho)$. 

\section{Laplacians with drift on Damek--Ricci spaces}
 
Let $X_0,X_1,\ldots,X_{n-1}$ be a frame of orthonormal left invariant vector fields on $S$ such that 
$X_0(e)=H$, $\{X_1(e),\ldots,X_{m_{\frak{v}}}(e)\}$ is an orthonormal basis of $\frak{v}$ 
and $\{X_{m_{\frak{v}}+1}(e),\ldots,X_{n-1}(e)\}$ is an orthonormal basis of $\frak{z}$.
Let $L_0$ be the corresponding Laplacian  $L_0= -\sum_{i=0}^{n-1} X_i^2$\,.

The operator $L_0$  has the following relationship with the Laplace--Beltrami operator ${\mathcal L}$ (see \cite[Proposition~2]{A1})
\begin{equation}\label{relationship}
\delta^{-1/2}L_0\,\delta^{1/2}f=\Big({\mathcal L}-\frac{Q^2}{4} \Big)f,
\end{equation}
for all smooth compactly supported radial functions $f$ on $S$.  
The spectra of ${\mathcal L}-Q^2/4$ on $L^2(\mu)$ and $L_0$ on $L^2(\rho)$ are both $[0,+\infty)$. For each bounded measurable function $m$ on $\mathbb{R}^+$ the operators $m({\mathcal L} -Q^2/4)$ and $m(L_0)$ are spectrally defined  and related by 
$$\delta^{-1/2}m(L_0)\,\delta^{1/2}f=m({\mathcal L} -Q^2/4)f\,,$$ 
for smooth compactly supported radial functions $f$ on $S$.
If ${k_{m(L_0)}}$ and ${k_{m({\mathcal L}-Q^2/4)}}$ denote the convolution kernels of $m(L_0)$ and $m({\mathcal L}-Q^2/4)$, then 
$$m({\mathcal L}-{Q}^2/4) f=f\ast{k_{m({\mathcal L}-Q^2/4)}}
\qquad{{\textrm{and}}}\qquad 
m(L_0) f=f\ast{k_{m(L_0)}}\qquad\forall f\in C^{\infty}_c(S)\,,$$
where $\ast$ denotes the convolution on $S$. We shall also write $m(L_0)\delta_e$ to denote $k_{m({L_0})}$. The following proposition is proved in  \cite{A1,ADY}.
\begin{proposition}\label{relazionenuclei}
Let $m$ be a bounded measurable function on $\mathbb{R}^+$. Then ${k_{m({\mathcal L}-Q^2/4)}}$ is radial and ${k_{m(L_0)}}=\delta^{1/2}\,{k_{m({\mathcal L}-Q^2/4)}}$. The spherical transform ${\mathcal H} k_{m({\mathcal L}-Q^2/4)}$ of $k_{m({\mathcal L}-Q^2/4)}$ is given by 
$${\mathcal H} k_{m({\mathcal L}-Q^2/4)}(\lambda)=m(\lambda^2),$$
for every positive real number $\lambda$.
\end{proposition}
Nontrivial positive characters of $S$ are given by $\chi_{\alpha}(X,Z,a)=a^{\alpha}$, $\alpha\in \RR\setminus\{0\}$. 
Let us consider the left invariant vector field $X_{\alpha}=\alpha \,  X_0$ and suppose that $\alpha \neq 0$. Reformulating the result by Hebisch, Mauceri and Meda \cite[Proposition 3.1]{HMM} 
in terms of left invariant vector fields and right measure, we get that the Laplacian with drift $L_{X}=L_0-X$ is symmetric with respect to some measure if and only if $X=\alpha X_0$. We shall denote by $L_\alpha$ the Laplacian with drift 
$$
L_\alpha=L_0-\alpha X_0\,.
$$
It is essentially 
self-adjoint on $L^2(\rho_{\alpha})$, where $\di\rho_{\alpha}=\chi_{\alpha}\,\di\rho$ 
and its spectrum is contained in the interval $[\alpha^2/4,\infty)$.  

\begin{remark}
A direct computation of the Laplace--Beltrami operator of Damek--Ricci spaces gives
$$
\mathcal L=L_0+QX_0=L_{-Q}\,.
$$
Thus the functional calculus for $L_{-Q}$ 
is equivalent to the functional calculus for the Laplace--Beltrami operator on $S$, which was studied in \cite{A, CS}. 
\end{remark}

\smallskip
 
Let $M$ be a bounded measurable function on $[\alpha^2/4,\infty)$. We shall give a sufficient condition on a function $M$ to be a $L^p$ spectral multiplier of $L_{\alpha}$. To do so, let us introduce some notation.  For every $p\in (1,\infty)\setminus\{2\}$ let $P_{\alpha,p}$ be the parabolic region  
$$P_{\alpha, p}=\left\{x+iy\in\CC:~x>\frac{y^2}{\alpha^2\,\sin ^2\phi_p^*}+\alpha^2/4\,\cos ^2\phi_p^* \right\}\,,  $$
where $\phi_p^*=\arcsin |2/p-1|$. For every positive numbers $W$ and $ \beta$ we denote by $\Sigma_W$ the complex strip defined by 
$\Sigma_W=\{x+iy\in\mathbb C: |y|<W\}$ and by $H^{\infty}(\Sigma_W;\beta)$ the space of bounded holomorphic functions $f$ on the strip 
$\Sigma_W$ such that
$$
|D^jf(s\pm iW)|\leq C\,(1+s^2)^{-j/2}\qquad \forall j=0,...,\beta,\forall s\in\mathbb R\,.
$$
We shall denote by $W_{\alpha,p}$ the number $|\alpha|\,|1/p-1/2| $.  Our main result is the following. 

\begin{theorem}\label{main}
Let $p\in (1,\infty)\setminus \{2\}$. Suppose that $M\in H^{\infty}(P_{\alpha,p})$ and that the function $M_{\alpha}$ 
defined by $M_{\alpha}(z)=M(z^2+\alpha^2/4)$ lies in 
$H^{\infty}(\Sigma_{W_{\alpha, p}};\beta)$, with $\beta>\max(2,n/2)$. 
Then $M(L_{\alpha})$ extends to a bounded operator on $L^p(\rho_{\alpha})$. 
\end{theorem}
\begin{proof}
By \cite[Proposition 4.1]{HMM}  for each $p$ in 
$ (1,\infty)$ the operator $M(L_{\alpha})$ 
is bounded on $L^p(\rho_{\alpha})$ if and only if $\chi_{\alpha}^{1/p-1/2}\, {M(  L_0+\alpha^2/4)}\delta_e$ is in $Cv_p(\rho)$.

Let $\omega$ be a smooth cutoff function supported in $[-1,1]$, equals $1$ in $[-1/4,1/4]$ such that 
$$
\sum_{h\in\mathbb Z} \omega(t-h)=1\qquad \forall t\in\mathbb R\,.
$$
For all $h\geq 2$, let 
$$
\omega_h(t)=\omega(t-h+1)+\omega(t+h-1).
$$
Note that ${\rm supp}\,\omega_h\subset [h-2,h]\cup [-h,-h+2]$.
We split the kernel of $M(L_0+\alpha^2/4)$ into a local and global part as follows:
$$
M(L_0+\alpha^2/4)\delta_e=\hat\eta \ast_{\mathbb R}M_{\alpha}(\sqrt L_0)\delta_e+
\widehat{(1-\eta)} \ast_{\mathbb R}M_{\alpha}(\sqrt L_0)\delta_e=\ell+g\,,
$$
where $\eta=\omega+\omega_2$. Define
\begin{equation}\label{lpgp}
\ell_p= \chi_{\alpha}^{1/p-1/2}\ell \qquad{\rm{and}}\qquad  g_p=\chi_{\alpha}^{1/p-1/2}g\,.
\end{equation}
In Proposition \ref{local} and Proposition \ref{global}, we shall prove that $\ell_p\in Cv_p(\rho)$ and $g_p\in Cv_p(\rho)$, respectively. This concludes the proof.
\end{proof}
As a direct consequence of the theorem above we compute the spectrum of $L_{\alpha}$ on $L^p(\rho_{\alpha})$, which we denote by $\sigma_p(L_\alpha)$. 
\begin{corollary}
Let $p\in (1,\infty)\setminus \{2\}$. The spectrum $\sigma_p(L_\alpha)$ is the parabolic region $P_{\alpha,p}$.
\end{corollary}
\begin{proof}
If $w \notin P_{\alpha,p}$, then the function $M(z)=(w-z)^{-1}$ satisfies the hypothesis of Theorem \ref{main}. Then $(w-L_{\alpha})^{-1}$ is bounded on $L^p(\rho_{\alpha})$ and $w$ belongs to the resolvent of $L_{\alpha}$ on $L^p(\rho_{\alpha})$. This proves that $\sigma_p(L_\alpha)\subseteq P_{\alpha, p}$. 

On the other hand, let $w\in P_{\alpha, p}$. Then the function $M(z)=(w-z)^{-1}$ is not holomorphic in $P_{\alpha, p}$. By \cite[Theorem 4.2]{HMM} the operator $M(L_{\alpha})=(w-L_{\alpha})^{-1}$ is not bounded on $L^p(\rho_{\alpha})$. Then $w\in\sigma_p(L_\alpha)$.
\end{proof}

We proceed next with the analysis of the local and global part of the kernel of $M(L_0+\alpha^2/4)$ constructed above separately. In particular, we note that on Lie groups of polynomial growth
\cite{HMM} the analysis of the global part of the kernel was based on the ultracontractivity 
of the heat semigroup generated by $L_0$ and a consequence of the Dunford--Pettis Theorem. However, since Damek--Ricci spaces are nonunimodular, the heat semigroup associated with $L_0$ is not ultracontractive, i.e., $e^{-tL_0}$ is not bounded from $L^1(\rho)$ to $L^{\infty}(\rho)$. Moreover,  Dunford--Pettis Theorem cannot be applied as in \cite{HMM}. To study the behavior of the global part of the kernel  we shall instead use tools from spherical analysis.

\subsection{Analysis of the local part.} Note that by the Fourier inversion formula
$$\ell=\hat{\eta}\ast_\RR M_{\alpha}(\sqrt  L_0)\delta_e=\Big(\frac{1}{2\pi}\int_{\RR}\eta(t)\,\hat{M}_{\alpha}(t)\,
\cos(t\sqrt L_0)\di t\Big)\delta_e\,.$$
Since $\eta$ is supported in $[-2,2]$, by finite propagation speed, $\ell$ is supported in $B(e,2)$.

We recall that a function $M:\RR^+\to \CC$ satisfies a mixed \MH condition of order $(s_0,s_{\infty})$ if
$$\max_{j=0,\ldots,s_0}\sup_{0<v<1}|v^j\,D^jM(v)|<\infty\,,$$
and
$$\max_{j=0,\ldots,s_{\infty}}\sup_{v\geq 1}|v^j\,D^jM(v)|<\infty\,.$$ In this case we say that $M\in {\rm{Horm}}(s_0,s_{\infty})$. 

\begin{proposition} \label{local}
The following hold:
\begin{itemize}
 \item[(i)] Let $s_0>3/2$ and $s_{\infty}>\max(3/2,n/2)$. If $M_{\alpha}\in {\rm{Horm}}(s_0,s_{\infty})$, then 
$\hat\eta \ast_{\mathbb R}M_{\alpha}(\sqrt L_0)$ is of weak type $1$ and bounded on $L^r(\rho)$ for all $r\in (1,\infty)$;
\item[(ii)] given $p\in (1,\infty)\setminus 2$, the function 
$\chi_{\alpha}^{ (1/p-1/2)}\,\hat\eta \ast_{\mathbb R}M_{\alpha}(\sqrt L_0)\delta_e$ is in  $Cv_p(\rho)$. 
\end{itemize}
\end{proposition}
\begin{proof}
As in \cite[Proposition 5.3]{HMM} we prove that if $M_{\alpha}\in {\rm{Horm}}(s_0,s_{\infty})$, then\\ $v\mapsto  
\hat\eta \ast_{\mathbb R}M_{\alpha}(\sqrt v)$ lies in ${\rm{Horm}}(s_0,s_{\infty})$. By \cite{V} this implies that 
$\hat\eta \ast_{\mathbb R}M_{\alpha}(\sqrt L_0)$ is of weak type $1$ and bounded on $L^r(\rho)$ for all $r\in (1,\infty)$. 
This implies tht $\ell=\hat\eta \ast_{\mathbb R}M_{\alpha}(\sqrt L_0)\delta_e$ is in $Cv_r(\rho)$ for all $r\in(1,\infty)$. 

Since the function $\ell$ is supported in $B(e,2)$ and $Cv_p(\rho)$ is a $C^{\infty}_c(S)$ module, we deduce that the function 
$\ell_p=\chi_{\alpha}^{ (1/p-1/2)}\,\,\ell$ is also in $Cv_p(\rho)$.

\end{proof}

\begin{remark}
Notice that the condition we require on $M_{\alpha}$ in Proposition \ref{local} is weaker than the condition required in the statement of Theorem \ref{main}. The holomorphy of the multiplier is only used in the study of the global part of the kernel.  
\end{remark}

\subsection{Analysis of the global part.} 
We decompose the global part of the kernel $g_p$ as 
\begin{equation}\label{decgp}
g_p=  \chi_{\alpha}^{ (1/p-1/2)}\,  \sum_{h=3}^{\infty}P_h(\sqrt  L_0)\delta_e=\sum_{h=3}^{\infty}g_{p,h}\,,
\end{equation}
where $\hat{P}_h=\omega_h\,\hat{M_{\alpha}}$.
By Proposition \ref{relazionenuclei} $P_h(\sqrt  L_0)\delta_e=\delta^{1/2}\,k_h$, 
where $k_h$ is the radial function whose spherical transform is $\mathcal H{k_h}=P_h$. We need the following estimates of the derivatives of the functions $\hat P_h$.  

\begin{lemma}\label{DPhhat}
Let $M_{\alpha}$ be in  $H^{\infty}(\Sigma_{W_{\alpha,p}};\beta)$ and $P_h$ be defined as above. For every $p,q\in \NN$ such that $1\leq p+q\leq \beta$ we have 
$$
\Big|\Big(-\frac{1}{\sinh s}\partial_s\Big)^{p}\Big(-\frac{1}{\sinh s}\partial_s\Big)^{q}\hat P_h(s)\Big|\lesssim e^{-(p+\frac{q}{2})\,s }\,s^{-\beta}\,e^{-W_{\alpha,p}\,s }
\qquad \forall s>1\,.
$$
\end{lemma}
\begin{proof}
By an induction argument, we may prove that there exist smooth functions $\psi_1,\dots,\psi_{p+q}$ bounded and with bounded derivatives of any order in $(1,+\infty)$, such that
\begin{equation}\label{inductiveformula}
\Big(-\frac{1}{\sinh s}\partial_s\Big)^{p}\Big(-\frac{1}{\sinh s}\partial_s\Big)^{q}\hat P_h(s)=e^{-(p+\frac{q}{2})\,s }\, \sum_{i=1}^{p+q}\psi_i(s) D^i \hat{P_h}(s)\qquad\forall s>1\,.
\end{equation}
Following the argument of \cite[Proposition 5.8]{HMM} we see that since $M_{\alpha}$ 
is in $H^{\infty}(\Sigma_{W_{\alpha,p}};\beta)$, then
\begin{equation}\label{derPh}
|D^k\hat{P}_h(t)|\leq C\,\|M_\alpha\|_{W_{\alpha,p};\beta}\,(1+|t| )^{-\beta}\,\nep^{-W_{\alpha,p}|t|}\qquad\forall t\in (h-2,h),\,\forall k\leq \beta\,,
\end{equation}
and $D^k\hat{P}_h(t)=0$ if $t\notin (h-2,h)$. 

We then apply estimate \eqref{derPh} in formula \eqref{inductiveformula} to complete the proof.
\end{proof}

\begin{proposition}\label{global}
If $M_{\alpha}$ 
is in $H^{\infty}(\Sigma_{W_{\alpha,p}};\beta)$ for $\beta>2$, then the function $g_p$ defined in \eqref{lpgp} is in $L^1(\rho)$. In particular, $g_p$ lies in $Cv_p(\rho)$ for all $p\in (1,\infty)$. 
\end{proposition}
\begin{proof}
We compute the $L^1$ norm of each function $g_{p,h}$:
\begin{equation}\label{gph}
\begin{aligned}
\|g_{p,h}\|_{L^1(\rho)}&=\int  \chi_{\alpha}^{ (1/p-1/2)}\,\,  \,|P_h(\sqrt  L_0)\delta_e|\di\rho\\
&=\int   a^{\alpha(1/p-1/2)}     \,\delta^{1/2}\,|k_h|\di\rho\\
&=\int   a^{\alpha(1/p-1/2)}     \,\delta^{-1/2}\,|k_h|\di\mu\\
&=\int       \,\delta^{-\frac{\alpha}{Q}\, (1/p-1/2)  -1/2}\,|k_h|\di\mu\\
&=\int \phi_{i\big(\alpha(1/p-1/2) \big)} \,|k_h |\,\di\mu\\
&=\int_0^{\infty} \phi_{i\big(\alpha(1/p-1/2) \big)}(r)\,|k_h(r)|\,A(r)\di r \,,
\end{aligned}
\end{equation}
where 
$A(r)$ is defined in formula \eqref{misura}.

The kernel $k_h$ can be computed by using the inverse Abel transform. We study the cases when $m_\frak{z}$ is either even or odd. 

\smallskip

{\bf{Case $m_\frak{z}$ even.}} In this case by formula \eqref{inv1}
$$
k_h(r)=\mathcal A^{-1}\mathcal F^{-1} P_h(r)=a^e_S\,\Big(-\frac{1}{\sinh r}\partial_r\Big)^{m_\frak{z}/2}\Big(-\frac{1}{\sinh r/2}\partial_r\Big)^{m_\frak{v}/2}\hat P_h(r)\,.
$$
Clearly $k_h(r)=0$ if $r \notin (h-2,h)$.  By Lemma \ref{DPhhat} if 
$r\in (h-2,h)$, then
$$
|k_h(r)|\lesssim e^{-\frac{Q}{2}\,r }\,r^{-\beta}\,e^{-W_{\alpha,p}\,r }\,.
$$
Then by \eqref{gph} we have
\begin{align*}
\|g_{p,h}\|_{L^1(\rho)}&\lesssim 
\int_{h-2}^{h}e^{W_{\alpha,p}\,r }\,e^{-\frac{Q}{2}\,r }\,(1+r)\,e^{-\frac{Q}{2}\,r }\,r^{-\beta}\,e^{-W_{\alpha,p}\,r }\,
e^{Q\,r }\,\di r\\
&\lesssim h^{1-\beta}\,,
\end{align*}
where we applied Lemma \ref{funzionisferiche} with $\lambda= i\alpha(1/p-1/2)$ and \eqref{pesoA}. By summing over $h\geq 3$ we obtain that $g_p\in L^1(\rho)$ if $\beta >2$.
\vskip0.2cm

{\bf{Case $m_\frak{z}$ odd.}} In this case by formula \eqref{inv2}
$$
k_h(r)=\mathcal A^{-1}\mathcal F^{-1} P_h(r)=a^o_S\,\int_r^{\infty} \Big(-\frac{1}{\sinh s}\partial_s\Big)^{\frac{m_\frak{z}+1}{2}} \Big(-\frac{1}{\sinh s/2}\partial_s\Big)^{\frac{m_\frak{v}}{2}}\hat P_h(s)\,
\frac{\sinh s}{\sqrt{\cosh s-\cosh r}}\,\di s\,.
$$
If $r\geq h$, then $k_h(r)=0$. If $r\in (h-2,h)$, then by Lemma \ref{DPhhat} 
\begin{align*}
|k_{h}(r)|&\lesssim \int_r^{h} e^{-\frac{Q+1}{2}\,s }\,s^{-\beta}\,e^{-W_{\alpha,p}\,s }\,\frac{e^s}{ \sqrt{\cosh s-\cosh r}}\di s\\
&\lesssim \int_r^{h} e^{-\frac{Q+1}{2}\,s }\,s^{-\beta}\,e^{-W_{\alpha,p}\,s }\,\frac{e^s}{e^{{r}/{2}}\,\sqrt{s- r}}\di s\\
&\lesssim e^{-\frac{Q+1}{2}r}\,r^{-\beta}\,e^{-W_{\alpha,p}\,r }\,e^{r/2}\int_0^2\frac{\di v}{\sqrt v}\\
&\lesssim e^{-\frac{Q}{2}\,r }\,r^{-\beta}\,e^{-W_{\alpha,p}\,r }\,.
\end{align*}
If $r<h-2$, then by Lemma \ref{DPhhat} 
\begin{align*}
|k_{h}(r)|&\lesssim \int_{h-2}^{h}  e^{-\frac{Q+1}{2}\,s }\,s^{-\beta}\,e^{-W_{\alpha,p}\,s }\frac{e^s}{e^{{r}/{2}}\,\sqrt{s- r}}\di s\\
&\lesssim e^{-\frac{Q+1}{2}\,h }\,h^{-\beta}\,e^{-W_{\alpha,p}\,h }\, e^h e^{-r/2}\int_{h-2}^h \frac{1}{\sqrt{s- r}}\di s\\
&\lesssim e^{-\frac{Q+1}{2}\,h }\,h^{-\beta}\,e^{-W_{\alpha,p}\,h }\, e^h e^{-r/2} \frac{1}{\sqrt{h-2-r}}.
\end{align*}

 By \eqref{gph}, Lemma \ref{funzionisferiche}, \eqref{pesoA} and the above estimates, we have
\begin{align*}
\|g_{p,h}\|_{L^1(\rho)}&\lesssim \int_0^{1} e^{-\frac{Q-1}{2}\,h }\,h^{-\beta-\frac{1}{2}}\,e^{-W_{\alpha,p}\,h }\,r^{n-1}\di r\\
&+
\int_1^{h-2}e^{W_{{\alpha,p}\,}r }\,e^{-\frac{Q}{2}\,r }\,(1+r)\, e^{-\frac{Q+1}{2}\,h }\,h^{-\beta}\,e^{-W_{\alpha,p}\,h }\,e^h e^{-r/2} \frac{1}{\sqrt{h-2-r}}\, e^{Q\,r }\di r\\
&+
\int_{h-2}^{h}e^{W_{\alpha,p}\,r }\,e^{-\frac{Q}{2}\,r }\,(1+r)\,e^{-\frac{Q}{2}\,r }\,r^{-\beta}\,e^{-W_{\alpha,p}\,r }\,
e^{Q\,r }\di r\\
&\lesssim e^{-\frac{Q-1}{2}\,h }\,h^{-\beta-\frac{1}{2}}\,e^{-W_{\alpha,p}\,h } +{\rm I} + h^{1-\beta}.
\end{align*}
We are left with the estimate of the summand
$$
{\rm I} =\int_1^{h-2}e^{W_{{\alpha,p}\,}r }\,e^{-\frac{Q}{2}\,r }\,(1+r)\, e^{-\frac{Q+1}{2}\,h }\,h^{-\beta}\,e^{-W_{\alpha,p}\,h }\,e^h e^{-r/2}  \frac{1}{\sqrt{h-2-r}}\, e^{Q\,r }\di r.
$$
By the change of variables $v=h-2-r$, we get
\begin{align*}
I &\lesssim h^{1-\beta} e^{-\frac{Q-1}{2}h}e^{-W_{\alpha,p}h}\int_{0}^{h-3}\frac{e^{(\frac{Q-1}{2}+W_{\alpha,p})(h-2-v)}}{\sqrt{v}}\di v\\
&\lesssim h^{1-\beta}.
\end{align*}
Therefore $\|g_{p,h}\|_{L^1(\rho)}\lesssim h^{1-\beta}$. 
By summing over $h\geq 3$ we obtain that $g_p\in L^1(\rho)$ if $\beta >2$. 
\end{proof}

\end{document}